\documentclass{article}
\usepackage{amssymb}
\usepackage{amsthm}
\usepackage[curve,matrix,arrow]{xy}
\theoremstyle{plain}\newtheorem{Theorem}{Theorem}[section]
\theoremstyle{plain}
\theoremstyle{plain}\newtheorem{Corollary}[Theorem]{Corollary}
\theoremstyle{plain}
\theoremstyle{plain}
\theoremstyle{plain}
\theoremstyle{definition}
\theoremstyle{definition}
\theoremstyle{definition}
\theoremstyle{definition}\newtheorem{Remark}[Theorem]{Remark}
\theoremstyle{definition}

\def\CA{{\mathcal{A}}}    \def\OG{{\mathcal{O}G}}  \def\OGb{{\mathcal{O}Gb}}
      
    \def\OP{{\mathcal{O}P}}
    \def\OQ{{\mathcal{O}Q}}
    
\def\CF{{\mathcal{F}}}    \def\OS{{\mathcal{O}S}}

\def\CO{{\mathcal{O}}}

\def\CA{{\mathcal{A}}}
\def\CA{{\mathcal{A}}}

\def\add{\mathrm{add}}

             \def\ten{\otimes}

\def\coMack{\mathrm{coMack}}    
\def\coMackbar{{\mathrm{coMack}_0}}    
           \def\tenkP{\otimes_{kP}}
\def\barE{\underline{E}}
\def\End{\mathrm{End}}           \def\tenkQ{\otimes_{kQ}}
\def\Endbar{\underline{\mathrm{End}}}

\def\Hom{\mathrm{Hom}}           
\def\Hombar{\underline{\mathrm{Hom}}}

\def\Id{\mathrm{Id}}             \def\tenA{\otimes_A}
             \def\tenB{\otimes_B}
\def\Ind{\mathrm{Ind}}

\def\Mack{\mathrm{Mack}}    
           \def\tenO{\otimes_{\mathcal{O}}}
\def\mod{\mathrm{mod}}
\def\modbar{\underline{\mathrm{mod}}}
\def\op{\mathrm{op}}

\def\pdim{\mathrm{pdim}}
\def\pr{\mathrm{pr}}
           \def\tenOP{\otimes_{\mathcal{O}P}}
         \def\tenOQ{\otimes_{\mathcal{O}Q}}
      
\def\Res{\mathrm{Res}}           \def\tenOR{\otimes_{\mathcal{O}R}}
         
           \def\tenOS{\otimes_{\mathcal{O}S}}
      
\def\stmod{\mathrm{stmod}}

\title{On equivalences for cohomological Mackey
functors} 
\author{Markus Linckelmann} 

\begin{document}

\maketitle

\begin{abstract}
By results of Rognerud, a source algebra equivalence between two 
$p$-blocks of finite groups induces an equivalence between the
categories of cohomological Mackey functors associated with these
blocks, and a splendid derived equivalence between two blocks
induces a derived equivalence between the corresponding
categories of cohomological Mackey functors. We prove this by
giving an intrinsic description of cohomological Mackey functors
of a block in terms of the source algebras of the block, and then
using this description to construct explicit two-sided tilting 
complexes realising the above mentioned derived equivalence. 
We show further that a splendid stable equivalence 
of Morita type between two blocks induces an equivalence between the 
categories of cohomological Mackey functors which vanish at the trivial 
group. We observe that the module categories of a block, the category
of cohomological Mackey functors, and the category of cohomological
Mackey functors which vanish at the trivial group arise in an
idempotent recollement. Finally, we extend a result of Tambara on the
finitistic dimension of cohomological Mackey functors to blocks.
\end{abstract}

\section{Introduction}

Let $p$ be a prime and $\CO$ a complete local principal ideal domain
with residue field $k=$ $\CO/J(\CO)$ of characteristic $p$; we allow 
the case $\CO=$ $k$.
Let $G$ be a finite group. The blocks of $\OG$ are the primitive
idempotents in $Z(\OG)$. By a result of Yoshida in \cite{YoshidaII}, 
the category $\coMack(\OG)$ of cohomological Mackey functors of $G$
with coefficients in $\mod(\CO)$ is equivalent to the category of
finitely generated modules over the algebra
$\End_{\OG}(\oplus_H\ \CO G/H)^\op$, where $H$ runs over the subgroups 
of $G$. This description implies that the category $\coMack(\OG)$ 
decomposes as a direct sum of the categories $\coMack(G,b)$, with $b$ 
running over the blocks of $\OG$, such that each $\coMack(G,b)$ is 
equivalent to the category of finitely generated modules over the 
algebra $\End_{\OG}(\oplus_H\ b\cdot \CO G/H)^\op$, where $H$ runs as 
before over the subgroups of $G$. By results of Th\'evenaz and Webb in 
\cite[\S 17]{TW}, this is the block decomposition of $\coMack(\OG)$.  
We use Yoshida's result to describe $\coMack(G,b)$ in an analogous way
at the source algebra level, obtaining as a consequence an alternative 
proof of a result of Rognerud in \cite{Rogn}, stating that blocks with
isomorphic source algebras have equivalent categories of cohomological 
Mackey functors. A {\it defect group} of a block $b$ of
$\OG$ is a maximal $p$-subgroup $P$ of $G$ such that $\OP$ is isomorphic
to a direct summand of $\OGb$ as an $\OP$-$\OP$-bimodule. A {\it source
idempotent} of $b$ is a primitive idempotent $i$ in the fixed point
algebra $(\OGb)^P$ with respect to the conjugation action of $P$ on
$\OGb$ such that $\OP$ is isomorphic to a direct summand of $i\OG i$ as
an $\OP$-$\OP$-bimodule. The algebra $A=$ $i\OG i$ with its canonical
image $Pi$ of $P$ in $A^\times$ is called a {\it source algebra of} $b$.

\begin{Theorem}\label{coMacksourcealgebra}
Let $G$ be a finite group, $b$ a block of $\OG$ with a defect group $P$
and let $A$ be a source algebra of $b$. Set $N=$ 
$\oplus_{Q}\ A\tenOQ \CO$, where in the direct sum $Q$ runs over the 
subgroups of $P$. Set $E=$ $\End_A(N)$.
There is an equivalence of categories $\coMack(G,b) \cong \mod(E^\op)$.
In particular, blocks with isomorphic source algebras have equivalent
categories of cohomological Mackey functors.
\end{Theorem}

By a result of Rognerud \cite[Proposition 4.5]{Rogn}, the last
statement holds more generally for {\it permeable Morita equivalences};
that is, Morita equivalences preserving $p$-permutation modules (and the 
proof presented below shows this as well).  
It is not necessary to take the direct sum over all subgroups of $P$ in 
the statement of \ref{coMacksourcealgebra}. 
Indeed, by \cite[Proposition 6.1]{Liperm}, every indecomposable direct 
summand of  $A\tenOQ O$ for some subgroup $Q$ of $P$ is isomorphic to a 
direct summand of $A\tenOR O$ for some fully centralised subgroup $R$ 
with respect to the fusion system $\CF$ determined by $A$ on $P$.
Thus we obtain a Morita equivalent endomorphism algebra if we
take the direct sum over a set of representatives of the
$\CF$-isomorphism classes of fully $\CF$-centralised subgroups of $P$.

\medskip
At the source algebra level, a splendid derived equivalence
between two source algebras $A$, $A'$ of blocks with a common
defect group $P$ is a derived equivalence induced by a {\it splendid 
Rickard complex of $A$-$A'$-bimodules $X$}; that is, $X$ is a bounded 
chain complex of $A$-$A$'-bimodules whose terms are direct sums of 
summands of the bimodules $A\tenOQ A'$, where $Q$ runs over the 
subgroups of $P$, such that we have homotopy equivalences 
$X\ten_{A'}X^*\simeq$ $A$ and $X^*\tenA X\simeq$ $A'$ as complexes of 
$A$-$A$-bimodules and of $A'$-$A'$-bimodules, respectively. Here $X^*$ 
denotes the $\CO$-dual of $X$, rewritten as a chain complex. Rognerud 
showed in \cite{Rogn} that a splendid Rickard equivalence between two
blocks induces a derived equivalence between the associated
categories of cohomological Mackey functors. We construct 
two-sided complexes which realise this derived equivalence.

\begin{Theorem} \label{coMackderived}
Let $A$, $A'$ be source algebras of blocks $b$, $b'$ of finite group 
algebras $\OG$, $\OG'$, respectively, with a common defect group $P$.
Set $N=$ $\oplus_Q\ A\tenOQ \CO$ and $N'=$ $\oplus_Q\ A'\tenOQ \CO$,
where $Q$ runs over the subgroups of $P$. Set 
$E=$ $\End_A(N)$ and $E'=$ $\End_{A'}(N')$. 
Let $X$ be a splendid Rickard complex of $A$-$A'$-bimodules.
Set $Y=$ $\Hom_A(X\ten_{A'} N', N)$ and
$Y'=$ $\Hom_{A'}(X^*\ten_{A} N, N')$.
Then $Y$ is a complex of $E$-$E'$-bimodules, $Y'$ is a complex
of $E'$-$E$-bimodules, the terms of $Y$ and $Y'$ are finitely
generated projective as left and as right modules, and we have 
homotopy  equivalences 
$$Y \ten_{E'} Y' \simeq E$$
$$Y'\ten_{E} Y \simeq E'$$
as complexes of $E$-$E$-bimodules and $E'$-$E'$-bimodules,
respectively. In particular, if $\OGb$ and $\CO G'b'$ are
splendidly Rickard equivalent, then the categories
$\coMack(G,b)$ and $\coMack(G',b')$ are derived equivalent.
\end{Theorem}

As mentioned above, the last statement of \ref{coMackderived} is due to 
Rognerud \cite[Theorem 5.8]{Rogn}, where - as before - this is shown to 
hold more generally for {\it permeable derived equivalences} (that is, 
derived equivalences induced by Rickard complexes which preserve 
complexes of $p$-permutation modules).  The first statement of 
\ref{coMackderived} is a special case of a more general construction of 
derived equivalences for endomorphism algebras, which we will describe 
in the next section (and which also allows for the extra generality 
regarding permeable derived equivalences).

\medskip
There are versions of the above results for stable endomorphism algebras 
and stable equivalences. Let $A$ be an $\CO$-algebra. An $A$-module
$U$ is {\it relatively $\CO$-projective} if $U$ is isomorphic to
a direct summand of $A\tenO W$ for some $\CO$-module $W$.
For any two $A$-modules $U$, $U'$ we denote by $\Hom_A^\pr(U,U')$ the
subspace of $\Hom_A(U,U')$ consisting of all $A$-homomorphisms which
factor through a relatively $\CO$-projective $A$-module.
We set $\Hombar_A(U,U')=$ $\Hom_A(U,U')/\Hom_A^\pr(U,U')$. 
We denote by $\stmod(A)$ the $\CO$-linear category having the
finitely generated $A$-modules as objects, with homomorphism spaces
$\Hombar_A(U,U')$ for any two finitely generated $A$-modules $U$, $U'$,
such that the composition of morphisms in $\stmod(A)$ is induced 
by the composition of $A$-homomorphisms (one checks that this is
well-defined). The resulting category is no longer abelian, but 
if $A$ is symmetric, then this category can be shown to be triangulated 
(a fact we will not need in this note).
For $G$ a finite group and $b$ a block of $\OG$, we denote by 
$\coMackbar(G,b)$ the full subcategory of $\coMack(G,b)$ 
consisting of all Mackey functors in $\coMack(G,b)$ whose evaluation at 
the trivial subgroup $\{1_G\}$ of $G$ is zero. 

\begin{Theorem} \label{coMackstable}
Let $G$ be a finite group, $b$ a block of $\OG$ with a defect group $P$
and let $A$ be a source algebra of $b$. 
Set $N=$ $\oplus_Q\ A\tenOQ \CO$, where $Q$ runs over the
subgroups of $P$. Set $\barE =$ $\Endbar_A(N)$.
There is an equivalence of categories $\coMackbar(G,b) \cong 
\mod(\barE^\op))$.
\end{Theorem}

\begin{Corollary} \label{recollementCor}
With the notation above, there is a recollement of abelian categories 
$$\xymatrix{\coMackbar(G,b) \ar[r] & \ar@<5pt>[l] \ar@<-5pt>[l]  
\coMack(G,b) \ar[r] & \ar@<5pt>[l] \ar@<-5pt>[l] \mod(\OGb) }$$
which is isomorphic to the idempotent recollement of the algebra $E$
with the idempotent $\tau$, where $\tau$ is the canonical
projection of $N$ onto its summand $A=$ $A\tenO \CO$.
\end{Corollary}

Let $A$, $A'$ be symmetric $\CO$-algebras. Following Brou\'e
\cite[\S 5]{BroueEq} an $A$-$A'$-bimodule $M$ is said to {\it induce a 
stable equivalence of Morita type} if $M$ is finitely generated 
projective as a left $A$-module, as a right $A'$-module, and if 
$M\ten_{A'} M^*\cong$ $A\oplus U$ for some projective $A$-$A$-bimodule 
and $M^*\tenA M\cong$ $A'\oplus U'$ for some projective 
$A'$-$A'$-bimodule $U'$. In that case, the functors $M\ten_{A'}-$ and 
$M^*\tenA-$ induce inverse equivalences $\modbar(A)\cong$ $\modbar(A')$. 
Stable equivalences of Morita type can be composed by tensoring the 
relevant bimodules together. In particular, a stable equivalence of 
Morita type between two block algebras induces a stable equivalence of 
Morita type between any two source algebras of these blocks.

\medskip
A stable equivalence of Morita type between two block algebras $\OGb$, 
$\CO G'b'$ is called {\it splendid} if it is induced by a 
$p$-permutation bimodule. In that case  the defect groups of $b$ and 
$b'$ are isomorphic, and for some identification of a defect group $P$ 
of $b$ with a defect group of $b'$ there is a choice of source algebras 
$A$ and $A'$ such that the induced stable equivalence of Morita type 
between $A$ and $A'$ is given by an $A$-$A'$-bimodule which is an 
indecomposable direct summand of $A\tenOP A'$.  By 
\cite[7.6.6]{Puigbook} the fusion systems of $A$ and $A'$ are then 
equal (see also \cite[3.1]{Lisplendid} for a proof of this fact). 

\begin{Theorem} \label{coMackstableeq}
Let $G$, $G'$ be finite groups, $b$ a block of $\OG$, and
$b'$ a block of $\CO G'$.
A splendid stable equivalence of Morita type
between $\OGb$ and $\CO Gb'$ induces an equivalence of categories
$$\coMackbar(G,b) \cong \coMackbar(G',b')\ .$$
\end{Theorem}

Let $\CA$ be an abelian category with enough projective objects.
The {\it projective dimension} $\pdim(U)$ of an object $U$ in $\CA$
is the smallest nonnegative integer $n$ for which there exists a 
projective resolution which is zero in all degrees larger than $n$, 
provided that $U$ has a bounded projective resolution, and 
$\pdim(U)=$ $\infty$ otherwise. The {\it finitistic dimension of $\CA$} 
is the supremum of the set of integers $\pdim(U)$, with $U$ running 
over all objects in $\CA$ having a finite projective dimension
(and with the convention that the finitistic dimension is $\infty$ if
this supremum is unbounded). The left or right finitistic dimension
of a finite-dimensional $k$-algebra is the finitistic dimension of its 
category of finitely generated left or right modules, respectively. 
The {\it $p$-sectional rank} $s(G)$ of a finite group $G$ is the 
maximum of the ranks of $p$-elementary abelian subquotients of $G$. 
Tambara showed in \cite{Tam} that $1+s(G)$ is equal to the finitistic 
dimension of $\coMack(kG)$. We show here, essentially by
adapting Tambara's proof, that there is an analogous statement for
block algebras.

\begin{Theorem} \label{coMackfin}
Let $G$ be a finite group, $b$ a block of $kG$ and $P$ a defect
group of $b$. The finitistic dimension of $\coMack(G,b)$ is equal
to $1+s(P)$.
\end{Theorem} 

\begin{Remark}
By results of Th\'evenaz and Webb in \cite[\S 17]{TW}, the category 
$\Mack(kG)$ of all Mackey functors on $G$ over $k$ has a block 
decomposition which corresponds bijectively to the block decomposition 
of $kG$. It seems however unknown whether the blocks of $\Mack(kG)$
admit a description in terms of the source algebras of the corresponding
blocks of $kG$. In contrast to Theorem \ref{coMackfin}, Bouc, Stancu 
and Webb \cite[Theorem 1.2]{BSW} showed that the finitistic dimension 
of $\Mack(kG)$ is zero. 
\end{Remark}

The proofs of the above theorems will be given in the 
last section. If not stated otherwise, modules are unital left modules.

\section{Derived equivalences and endomorphism algebras}

An $\CO$-algebra $A$ is called {\it symmetric} if $A$ is isomorphic
to its $\CO$-dual $A^*$ as an $A$-$A$-bimodule. If $s\in$ $A^*$ is the
image of $1_A$ under some bimodule isomorphism $A\cong$ $A^*$, then
$s$ is called a {\it symmetrising form on} $A$. Matrix algebras, 
finite group algebras over $\CO$, their block algebras and source 
algebras are symmetric. 

For $A$ an $\CO$-algebra, $(U,\delta)$ a complex of right
$A$-modules and $(V,\epsilon)$ a complex of left $A$-modules we denote
by $U\tenA V$ the complex of $\CO$-modules which in degree $n$
is equal to $\oplus_{(i,j)}\ U_i\tenA V_j$, with $(i,j)$ 
running over the set of pairs of integers satisfying
$i+j=$ $n$. The differential of $U\tenA V$ is the sum of the
maps $\Id_{U_i}\ten \epsilon_j$ and $(-1)^i\delta_i\ten \Id_{V_j}$ from
$U_i\tenA V_j$ to $U_i\tenA V_{j-1}$ and to $U_{i-1}\tenA V_j$.  
If $U$, $V$ are bimodule complexes, then $U\tenA V$ inherits those
bimodule structures.

For $(V,\delta)$ and $(W,\zeta)$ two complexes of $A$-modules,
we denote by $\Hom_A(V,W)$ the complex of $\CO$-modules
which in degree $n$ is equal to $\oplus_{(i,j)}\ \Hom_A(V_i,W_j)$,
with $(i,j)$ runnig over the pairs of integers satisfying
$j-i=$ $n$. The differential is the sum of the maps from
$\Hom_A(V_i,W_j)$ to $\Hom_A(V_i,W_{j-1})$ and to
$\Hom_A(V_{i+1}, W_{j})$, where the first is given by 
composition with $\zeta_j$, and the second by precomposition
with $(-1)^{n+1}\epsilon_{i+1}$. As before, if $V$, $W$ have
bimodule structures, then these induce bimodule structure on
the complex $\Hom_A(V,W)$. Note that this sign convention implies
that the $k$-dual $V^*$ of the complex $V$ is the complex which
in degree $n$ is equal to $V^*_{-n}$, with differential
$V^*_{-n}\to$ $V^*_{-n+1}$ given by precomposing with
$(-1)^{n+1}\epsilon_{-n+1} : V_{-n+1}\to$ $V_{-n}$.
If $A$ is symmetric and $s\in$ $A^*$ is a symmetrising form of $A$,
then composition with $s$ induces an isomorphism of complexes
$\Hom_A(V,A)\cong$ $V^*$.

For $A$, $A'$ symmetric $k$-algebras, a {\it Rickard complex
of $A$-$A'$-bimodules} is a bounded complex $X$ of $A$-$A'$-bimodules
whose terms are finitely generated projective as left $A$-modules
and as right $A'$-modules, such that we have homotopy
equivalences $X\ten_{A'} X^*\simeq$ $A$ and $X^*\tenA X\simeq$
$A'$ as complexes of $A$-$A$-bimodules and $A'$-$A'$-bimodules,
respectively.

\begin{Theorem} \label{endoderived}
Let $A$, $A'$ be symmetric $k$-algebras. Let $N$ be a finitely
generated $A$-module, and let $N'$ be a finitely generated
$A'$-module. Set $E=$ $\End_A(N)$ and $E'=$ $\End_{A'}(N')$.
Let $X$ be a Rickard complex of $A$-$A'$-bimodules. Suppose that
every indecomposable direct summand of every term of $X\ten_{A'} N'$
is isomorphic to a direct summand of $N$ and that every
indecomposable direct summand of every term of $X^*\tenA N$ is
isomorphic to a direct summand of $N'$. 
Set $Y=$ $\Hom_{A}(X\ten_{A'} N', N)$ and
$Y'=$ $\Hom_{A'}(X^*\tenA N, N')$. 
Then $Y$ is a complex of $E$-$E'$-bimodules, $Y'$ is a complex
of $E'$-$E$-bimodules, the terms of $Y$ and $Y'$ are finitely
generated projective as left and as right modules, and we have 
homotopy  equivalences 
$$Y\ten_{E'} Y' \simeq E$$
$$Y'\ten_{E} Y \simeq E'$$
as complexes of $E$-$E$-bimodules and $E'$-$E'$-bimodules,
respectively. In particular, $E$ and $E'$ are
derived equivalent. Moreover, if $X$ and $X^*$ induce a Morita
equivalence between $A$ and $A'$, then $Y$ and $Y'$ induce
a Morita equivalence between $E$ and $E'$.
\end{Theorem}

\begin{proof}
If $i$, $j$ are idempotents in $A$, then we have a canonical 
isomorphism $iA\tenA Aj\cong$ $iAj$, induced by multiplication in $A$. 
We translate this to endomorphism algebras. Let  $U$, $V$ be 
direct summands of $N$. Let $\tau$ be a projection of $N$ onto $U$, 
viewed as an idempotent in $E$, and let $\sigma$ be a projection of
$N$ onto $V$. 
Then $\Hom_A(U,N)\cong$ $E\circ\tau$ is a projective left $E$-module 
via composition of homomorphisms. Similarly, $\Hom_A(N,V)\cong$ 
$\sigma\circ E$ is a projective right $E$-module via precomposition by 
homomorphisms in $E$. 
Thus the first observation applied to $E$ instead of $A$
implies that composition of homomorphisms induces an isomorphism 
$$\Hom_A(N, V)\ten_E\Hom_A(U,N) \cong \Hom_A(U,V)$$
which is natural and additive in $U$ and $V$. Thus this isomorphism holds
for $U$ and $V$ finite direct sums of direct summands of $N$.
Similar statements hold for $A'$ and $N'$.
By the assumptions, the indecomposable direct summands of the terms of 
$X^*\ten_{A} N$ are isomorphic to direct summands of $N'$. Thus the above 
isomorphism, applied with $U=$ $V=$ $X^*\ten_{A} N$ and $N'$ instead of $N$, 
shows that composition of homomorphisms yields an isomorphism
$$\Hom_{A'}(N', X^*\ten_{A} N) \ten_{E'} \Hom_{A'}(X^*\ten_{A} N, N')
\cong \End_{A'}(X^*\ten_A N)$$
as graded $E$-$E$-bimodules. A straightforward verification
(checking signs) shows that this isomorphism commutes with the 
differentials, so this is an isomorphism of complexes of 
$E$-$E$-bimodules.
By a standard adjunction, the complex on the left side is isomorphic
to $\Hom_A(X\ten_{A'} N', N) = Y$. 
By another standard adjunction, the right side is isomorphic to the 
complex
$$\Hom_A(N, X\ten_{A'} X^* \tenA N)$$
Since $X$ is a Rickard complex, we have
$X\ten_{A'} X^* \ten_A N\simeq N$, and hence 
$$\Hom_A(N, X\ten_{A'} X^* \ten_{A} N)\simeq \Hom_A(N,N)=E$$
This shows that $Y\ten_{E'} Y'\simeq$ $E$. A similar argument shows 
that $Y'\ten_{E} Y \simeq E'$, which completes
the proof.
\end{proof}

\section{On the finitistic dimension of endomorphism algebras}

The result of this section holds regardless of the characteristic
of the field $k$.  Let $A$ be a finite-dimensional $k$-algebra and
$U$ a finitely generated $A$-module. We denote by $\add(U)$ the
full additive subcategory of all $A$-modules which are isomorphic to
finite direct sums of direct summands of $U$. 
We use without further reference the following standard facts. 
If $V$ belongs to $\add(U)$, then $\Hom_A(U,V)$ is a projective right 
$\End_A(U)$-module, and any finitely generated projective right
$\End_A(U)$-module is of this form, up to isomorphism.  Given two 
idempotents $i$, $j$ in $A$, every homomorphism of right $A$-modules 
$iA\to$ $jA$ is induced by left multiplication with an element in $jAi$. 
Multiplication by $i$ is exact; in particular, if $Z$ is a complex of 
$A$-modules which is exact or which has homology concentrated in a
single degree, the same is true for finite direct sums of complexes
of the form $iZ$. Translated to endomorphism algebras this implies that 
for any two $A$-modules $V$, $W$ in $\add(U)$, any homomorphism of right 
$\End_A(U)$-modules $\Hom_A(U,V)\to$ $\Hom_A(U,W)$ is induced by 
composition with an $A$-homomorphism $V\to$ $W$. Thus any complex of 
finitely generated projective right $\End_A(U)$-modules is isomorphic to 
a complex obtained from applying the functor $\Hom_A(U,-)$ to a complex of
$A$-modules $Z$ whose terms belong to $\add(U)$. Moreover, if $\Hom_A(U,Z)$ is
exact or has homology concentrated in a single degree, the same is true
for any complex of the form $\Hom_A(U',Z)$, where $U'\in$ $\add(U)$.

Let $A$ and $B$ be symmetric $k$-algebras. 
Following \cite{LinHecke}, an $A$-$B$-bimodule $M$ is said
to induce a {\it separable equivalence between $A$ and $B$} if
$M$ is finitely generated projective as a left $A$-module, as a 
right $B$-module, and if $A$ is isomorphic to a direct summand of
$M\tenB M^*$ as an $A$-$A$-bimodule and $B$ is isomorphic to 
$M^*\tenA M$ as  a $B$-$B$-bimodule. If there is such a bimodule,
we say that $A$ and $B$ are {\it separably equivalent.} 
The link between this section and block theory comes from the
fact that a block algebra of a finite group is separably equivalent to 
its defect group algebras via a $p$-permutation bimodule. Since the
notion of separable equivalence is transitive, any two block algebras 
with isomorphic defect groups are separably equivalent via 
$p$-permutation bimodules.

\begin{Theorem} \label{endofin}
Let $A$, $B$ be symmetric $k$-algebras. Let $M$ be an 
$A$-$B$-bimodule inducing a separable equivalence between $A$ and $B$.
Let $U$ be a finitely generated $A$-module and $V$ a finitely generated 
$B$-module. Suppose that $M^*\tenA U\in$ $\add(V)$ and that 
$M\tenB  V\in$ $\add(U)$. Then the right finitistic dimensions of 
$\End_A(U)$ and of $\End_B(V)$ are equal.
\end{Theorem}

\begin{proof} 
Set $E=$ $\End_A(U)$ and $F=$ $\End_B(V)$. We show first that
$\add(M\tenB V)=$ $\add(U)$ and $\add(M^*\tenA U)=$ $\add(V)$.
By the assumptions, we have $\add(M\tenB V)\subseteq$ $\add(U)$.
Thus $\add(M^*\tenA M \tenB V)\subseteq$ $\add(M^*\tenA U)\subseteq$
$\add(V)$. Since $B$ is isomorphic to a direct summand of the
$B$-$B$-bimodule $M^*\tenA M$, it follows that $V$ is isomorphic
to a direct summand of $M^*\tenA M\tenB V$. Thus the previous
inclusions of additive categories are equalities. The same
argument with reversed roles shows the second equality.
We show next that for any finitely generated right $F$-module
of finite projective dimension there is a finitely generated
right $E$-module with the same pojective dimension.
Let $Y$ be a right $F$-module with a finite projective dimension
$\pdim(Y)=$ $n$. By the remarks at the beginning of this
section, any such resolution is of the form
$$0\to\Hom_B(V,V_n)\to\cdots\to\Hom_B(V,V_1)\to\Hom_B(V,V_0)\to Y\to 0$$
obtained from applying the functor $\Hom_B(V,-)$ to a complex
of $B$-modules
$$Z = \cdots\to 0\to V_n\to\cdots\to V_1\to V_0\to 0 \to \cdots$$
whose components $V_i$ are in $\add(V)$, together with a surjective 
homomorphism of right $F$-modules $\Hom_B(V,V_0)\to$ $Y$. 
Thus the complex $\Hom_B(V,Z)$ has homology concentrated in degree $0$, 
and the homology in degree $0$ is $Y$.  By the remarks at the beginning 
of this section, it follows that for any $B$-module $V'$ in $\add(V)$, 
applying the functor $\Hom_B(V',-)$ to the complex of $B$-modules $Z$ 
yields a complex of $k$-vector spaces
$$\Hom_B(V',Z)$$
which has nonzero homology only in degree $0$, if any.
By the assumptions, this applies in particular to the $B$-module
$V'=$ $M^*\tenA M\tenB V$. Using a standard adjunction, we obtain
a complex of the form
$$\Hom_A(M\tenB V, M\tenB Z)$$
which has homology concentrated in degree $0$.
This is a complex of right $\End_A(M\tenB V)$-modules.
Since $\add(M\tenB V)=$ $\add(U)$, it follows that the complex
$$\Hom_A(U, M\tenB Z)$$
has homology concentrated in degree $0$. This is a complex of right 
$E$-modules. Its homology in degree $0$ is therefore a right
$E$-module $X$ with projective dimension at most $n$. We show that
$X$ has projective dimension exactly $n$. If $n=0$, there is nothing
to prove, so assume that $n>0$. The fact that $\pdim(Y)=n$ implies
in particular that the first nonzero differential 
$\Hom_B(V,V_n)\to\Hom_B(V,V_{n-1})$ in the complex $\Hom_B(V,Z)$
is not split injective. Thus the first nonzero differential 
$V_n\to$ $V_{n-1}$ in $Z$ is not split injective. But then the first
nonzero differential of $M\tenB Z$ is not split injective either,
for if it were, so would be that of $M^*\tenA M\tenB Z$, which is
not possible, as $Z$ is a direct summand, as a complex of $B$-modules,
of $M^*\tenA M\tenB Z$. This in turn implies that the first nonzero
differential of $\Hom_A(U, M\tenB Z)$ is not split injective as a right
$E$-module homomorphism.  This shows that $\pdim(X)=$ $\pdim(Y)$.
Thus the right finitistic dimension of $E$ is at least as large as
that of $F$. Exchanging the roles of $E$ and $F$ shows the equality
of right finitistic dimensions as stated.
\end{proof}

\section{Proofs}

Let $A$ be an $\CO$-algebra which is finitely generated as an $\CO$-module.
Let $U$, $U'$ be finitely generated $A$-modules. An idempotent in 
$\End_A(U)$ corresponds to a projection of $U$ onto a direct summand 
of $U$. The Krull-Schmidt theorem implies that this correspondence 
induces a bijection between conjugacy classes of primitive idempotents 
in $\End_A(U)$ and isomorphism classes of indecomposable direct 
summands of $U$. Thus if $U$ and $U'$ have the same isomorphism classes 
of indecomposable direct summands, or equivalently, if $\add(U)=$
$\add(U')$, then $\End_A(U)$ and $\End_A(U')$ 
are Morita equivalent. We will use this well-known fact without
further reference. 

\begin{proof}[Proof of Theorem \ref{coMacksourcealgebra}]
Let $G$ be a finite group, $b$ a block of $\OG$, $P$ a defect
group of $b$, and $i\in$ $(\OGb)^P$ a source idempotent of $b$.
By \cite[3.5]{Puigpoint}, multiplication with $i$ induces a Morita 
equivalence between $kGb$ and the source algebra $A=$ $i\OG i$. 
By Yoshida's theorem \cite[4.3]{YoshidaII} mentioned above, the
category $\coMack(G,b)$ is equivalent to 
$\mod(\End_{\OGb}(\oplus_H\ b\cdot \CO G/H)^\op)$, where $H$ runs over
the subgroups of $G$. Fix a subgroup $H$ of $G$, and let $U$ be an 
indecomposable direct summand of $b\cdot \CO G/H\cong$
$b\cdot\Ind^G_H(\CO)$.
Then $U$ has trivial source. By \cite[6.1]{Likleinfour}
there exists a vertex $Q$ of $U$ such that $U$ is isomorphic to
a direct summand of $\OG i\tenkQ \CO\cong$ $\OG i\ten_A A\tenkQ \CO$. 
Thus $U\cong$ $\OG i\tenA V$ for some indecomposable direct summand $V$ 
of $A\tenOQ \CO$. Conversely, if $W$ is an
indecomposable direct summand of $N$, then $\OG i\tenA W$ is isomorphic 
to a direct summand of $\CO G/Q$ for some subgroup $Q$ of $P$.
This implies that $\End_{\OGb}(\oplus_H\ b\cdot \CO G/H)^\op$ is Morita 
equivalent to $\End_{\OGb}(\OG i\tenA N)^\op$, 
where $H$ runs over the subgroups of $G$.
Since multiplication by $i$ induces a Morita equivalence between
$\OGb$ and $A$ it follows that the last algebra is isomorphic
to $\End_A(N)^\op=$ $E^\op$, whence the equivalence $\coMack(G,b)\cong$ 
$\mod(E^\op)$. Let $A'$ be a source algebra of a block $b'$ of a finite
group $G'$ having $P$ as a defect group. Set $N'=$ $\oplus_Q A'\tenOQ \CO$, 
where $Q$ runs over the subgroups of $P$. By the above, a Morita 
equivalence between $A$ and $A'$ which induces a bijection between the 
isomorphism classes of indecomposable direct summand of $N$ and of $N'$
induces an equivalence $\coMack(G,b)\cong$ $\coMack(G',b')$.
\end{proof}

\begin{proof}[{Proof of Theorem \ref{coMackderived}}]
We show that this is a special case of \ref{endoderived}.
Let $Q$ be a subgroup of $P$. We need to show that every 
indecomposable direct summand of one of the terms of the complex
of $A$-modules $X\ten_{A'} A\tenOQ \CO\cong$ $X\tenOQ \CO$ is isomorphic 
to a direct summand of $A\tenOR \CO$ for some subgroup $R$ of $P$.
Since the terms of $X$ are finite direct sums of summands of bimodules
of the form $A\tenOS A'$, with $S$ running over the subgroups of 
$P$, it suffices to show that for $S$ a subgroup of $P$, any 
indecomposable direct summand $U$ of the left $A$-module
$A\tenOS A' \tenkQ \CO$ is isomorphic 
to a direct summand of $A\tenOR \CO$ for some subgroup $R$.
Since $A'$ is a permutation $\OP$-$\OP$-module which is finitely
generated projective as a left and right $\OP$-module, it follows
that any indecomposable direct summand of $A'$ as an
$\OQ$-$\OS$-bimodule is isomorphic to $\OS \tenOR (_\psi{\OQ})$
for some injective group homomorphism $\psi : R\to$ $Q$. (See
\cite[Appendix]{Lisplendid} for a review on the use of the
bimodule structure of source algebras, going back to work of 
Puig \cite{Puigloc}.) Thus $U$ is isomorphic to a direct summand of a 
left $A$-module of the form 
$A\tenOS \OS\tenOR (_\psi{\OQ})\tenOQ\CO\cong$
$A\tenOR\CO$ as claimed. Exchanging the roles of $A$ and $A'$
completes the proof of the first statement. Let last statement
follows from this, together with the description in 
\ref{coMacksourcealgebra} of cohomological Mackey 
functors in terms of source algebras. 
\end{proof}

\begin{proof}[{Proof of Theorem \ref{coMackstable}}]
Let $G$ be a finite group, $b$ a block of $\OG$, and $P$ a defect
group of $b$. Let $N$ and $\barE$ as in the statement.
Since $N$ is $\CO$-free, it follows that an 
$A$-endomorphism $\varphi$ of $N$ factors through a relatively 
$\CO$-projective $A$-module if and only if it factors through a 
projective $A$-module. This is in turn equivalent to stating
that $\varphi$ is a sum of endomorphisms which factor
through the summand $A\tenO \CO$ of $N$. It follows that 
$\End^\pr_A(N)$ is equal to the ideal in $E=$ $\End_A(N)$ generated 
by the projection $\tau$ of $N$ onto $A\tenO \CO$. 
Thus an $E$-module $U$ is the inflation to $E$ from
an $\barE$-module if and only if $\tau$ annihilates $U$.
By Yoshida's translation of $E^\op$-modules into cohomological Mackey 
functors in \cite{YoshidaII}, the $E^\op$-modules annihilated by $\tau$ 
correspond exactly to the cohomological Mackey functors which vanish at 
the trivial subgroup. This shows the equivalence of categories 
$\coMackbar(G,b)\cong$ $\mod(\barE^\op)$. 
\end{proof}

\begin{proof}[{Proof of Corollary \ref{recollementCor}}]
With the notation of the previous proof, we have $\barE = $ 
$E/E\circ \tau \circ E$. We also have $\tau\circ E\circ \tau =$ 
$\End_A(A)\cong$ $A^\op$. Thus we have a standard recollement 
determined by the idempotent $\tau$ in $E$.
The result follows, using the Morita equivalences between $A$, 
$\mod(E^\op)$, $\mod(\barE^\op)$ and $\OGb$, $\coMack(G,b)$,
and $\coMackbar(G,b)$ from \cite[3.5]{Puigpoint}, 
\ref{coMacksourcealgebra}, \ref{coMackstable}, respectively.
\end{proof}

We will need the following variation of a fact stated at
the beginning of this section: the idempotents in the kernel of the 
canonical algebra homomorphism $\End_A(U)\to$ $\Endbar_A(U)$ correspond 
exactly to the relatively $\CO$-projective indecomposable direct 
summands of $U$. Thus if two finitely generated $A$-modules $U$ and 
$U'$ have the same isomorphism classes of indecomposable 
direct summands which are not relatively $\CO$-projective, then 
$\Endbar_A(U)$ and $\Endbar_A(U')$ are Morita equivalent. 

\begin{proof}[{Proof of Theorem \ref{coMackstableeq}}]
Suppose that there is a splendid stable equivalence of Morita
type between $\OGb$ and $\CO G'b'$. As mentioned in the introduction,
it follows from \cite[7.6.6]{Puigbook} or \cite[3.1]{Lisplendid} that 
there is an identification of $P$ with a defect group of $b'$ and that 
there are source algebras $A$ of $\OGb$ and $A'$ of $\CO G'b'$ such 
that the induced stable equivalence of Morita type between $A$ and 
$A'$ is given by a bimodule $M$ which is an indecomposable direct 
summand of $A\tenOP A'$.  
Set $N=$ $\oplus_Q\ A\tenOQ \CO$ and $N'=$ $\oplus_Q\ A'\tenOQ \CO$, 
where in the direct sums $Q$ runs over the subgroups of $P$. 
Set $E=$ $\End_{A}(N)$ and $\barE=$ $\Endbar_{A}(N)$. Similarly, 
Set $E'=$ $\End_{A'}(N')$ and $\barE'=$ $\Endbar_{A'}(N')$. By 
\ref{coMackstable}, in order to show that $\coMackbar(G,b)\cong$ 
$\coMackbar(G',b')$ it suffices to show that $\mod(\barE)\cong$ 
$\mod(\barE')$. Since $M$ induces a stable equivalence of Morita
type, it follows that we have an algebra isomorphism
$\Endbar_A(M\ten_{A'} N') \cong \barE'$. 
Thus it suffices to show that the $A$-modules $N$ and $M\ten_{A'} N'$
have the same isomorphism classes of indecomposable nonprojective 
direct summands.

Let $V'$ be an indecomposable nonprojective direct summand of $N'$.
Then there is a subgroup $Q$ of $P$ such that $V'$ is an indecomposable
nonprojective direct summand of $A'\tenOQ V'$. It follows that 
$M\ten_{A'} V'$ has, up to isomorphism, a unique indecomposable
nonprojective direct summand $V$, which is by construction
isomorphic to a direct summand of $A\tenOP A'\tenOQ \CO$.
The argument used in the proof of Theorem
\ref{coMackderived} above shows that $V$ is a summand of
$A\tenOR\CO$ for some subgroup $R$ of $P$, hence $V$ is isomorphic
to a summand of $N$. Exchanging the roles of $N$ and $N'$ as
well as $M$ and $M^*$ implies that the functors $M\ten_{A'}-$
and $M^*\tenA-$ induce a bijection between the isomorphism 
classes of indecomposable nonprojective direct summands of $N$ and 
$N'$. The result follows.
\end{proof}

\begin{proof}[{Proof of Theorem \ref{coMackfin}}]
Denote by $A$ a source algebra of $kGb$ for the defect group $P$.
It is well-known that $A$ and the defect group algebra $kP$
are separably equivalent via the $A$-$kP$-bimodule $A_P$, obtained
from restricting the $A$-$A$-bimodule $A$ on the right side to $kP$ 
via the canonical injection $kP\to$ $A$ (see e.g. \cite[4.2]{Liperm}). 
Since $A$ is symmetric, the dual of $A_P$ is isomorphic to ${_P{A}}$, 
obtained from restricting the $A$-$A$-bimodule $A$ to $kP$ on the left 
side. By Tambara's main result in \cite{Tam}, the finitistic dimension of 
$\coMack(kP)$ is equal to $1+s(P)$. Set $V=$ $\oplus_{Q\leq P}\ kP/Q$ 
and $U=$ $A\tenkP V=$ $\oplus_{Q\leq P}\ A\tenkQ k$. We have 
$({_P{A})\tenA U}=$ $\Res_P(U)$.
This is a permutation $kP$-module, because $A$ has a $P\times P$-stable
$k$-basis. Thus the functor $A\tenkP-$ sends $V$ to $\add(U)$, and
the functor ${_P{A}}\tenA - $ sends $U$ to $\add(V)$.  It follows
from Theorem \ref{endofin} that $\End_{kP}(V)$ and $\End_A(U)$ have
the same right finitistic dimension. Since the categories of 
finitely generated right modules over $\End_{kP}(V)$ and $\End_A(U)$ 
are equivalent to $\coMack(kP)$ and $\coMack(G,b)$, respectively,
by Yoshida's theorem \cite[4.3]{YoshidaII} and \ref{coMacksourcealgebra} 
above, it follows that the finitistic dimension of $\coMack(G,b)$ is 
$1+s(P)$ as stated.
\end{proof}


\end{document}